\documentclass[preprint,12pt]{elsarticle}

\usepackage{epsfig,amsmath,amssymb,amsthm,color,amsfonts,epstopdf,mathrsfs}
\biboptions{numbers,sort&compress}

\everymath{\displaystyle}
\def\l{\left}
\def\r{\right}
\def\CC{{\mathbb{C}}}

\def\RR{{\mathbb{R}}}
\def\NN{{\mathbb{N}}}

\def\OO{{\cal{O}}}

\def\R#1{$(\ref{#1})$}

\newcommand{\bb}[1]{\begin{equation}\label{#1}}
\newcommand{\ee}{\end{equation}}
\newcommand{\bbb}{\begin{eqnarray}}
\newcommand{\eee}{\end{eqnarray}}
\newcommand{\bbbb}{\begin{eqnarray*}}
\newcommand{\eeee}{\end{eqnarray*}}

\newcommand{\nnn}{\nonumber}
\newcommand{\no}{\noindent}

\newtheorem{thm}{Theorem}
\newtheorem{lemma}{Lemma}

\theoremstyle{remark}
\newtheorem{rem}{Remark}
\theoremstyle{definition}
\newtheorem{define}{Definition}

\newcommand{\clearallnum}{
    \numberwithin{equation}{section} \setcounter{equation}{0}
    \numberwithin{thm}{section} \setcounter{thm}{0}
    \numberwithin{lemma}{section} \setcounter{lemma}{0}
    \numberwithin{cor}{section} \setcounter{cor}{0}
    \numberwithin{rem}{section} \setcounter{rem}{0}
    \numberwithin{define}{section} \setcounter{define}{0}}

\journal{}

\begin{document}

\begin{frontmatter}



\title{The quenching of solutions to time-space fractional Kawarada problems}


\author{Joshua L. Padgett}
\ead{joshua.padgett@ttu.edu}

\address{Texas Tech University, Department of Mathematics and Statistics, Broadway and Boston, Lubbock, TX 79409-1042}

\begin{abstract}
Quenching solutions to a Kawarada problem with a Caputo time-fractional derivative and a fractional Laplacian are considered. The solutions to such problems may only exist locally in time when quenching occurs. Quenching and non-quenching solutions are shown to remain positive and be monotonically increasing in time under minor restrictions. Conditions for quenching to occur are demonstrated and shown to depend on the domain size.
\end{abstract}

\begin{keyword}



Kawarada problem \sep quenching solution \sep Caputo derivative \sep fractional Laplacian \sep local existence and uniqueness \sep positivity and monotonicity

\end{keyword}

\end{frontmatter}


\section{Introduction} \clearallnum

The purpose of this paper is to investigate the quenching phenomenon of a time-space fractional semilinear equation. Let $\Omega$ be an open bounded domain in $\RR^d$ with smooth boundary $\partial\Omega.$ We then define $Q_T \mathrel{\mathop:}= \Omega\times (0,T)$ and the parabolic boundary $\Gamma_T = \partial\Omega\times (0,T).$ Consider the following nonlocal Kawarada problem:
\bb{b1}
\l\{\begin{aligned}
\partial_t^\alpha u & = -(-\Delta)^s u + f(u), & (x,t) &\in Q_t,\\
u & = 0, & (x,t) &\in \Gamma_T,\\
u(x,0) & = u_0(x), & x &\in \Omega,
\end{aligned}\r.
\ee
where $\partial_t^\alpha$ denotes the Caputo time-fractional derivative of order $\alpha\in(0,1),$ and $(-\Delta)^s$ is the fractional Laplacian with $s\in(0,1),$ and the continuous initial data $u_0\,:\,\Omega\to \RR^+$ is such that $0 \le u_0 \ll c.$ The nonlinear reaction term $f\,:\,B_\rho\to \RR^+,$ where $0<\rho< c$ and $B_\rho\mathrel{\mathop:}=\{u\in L^\infty(\Omega)\,:\, \|u\|_\infty < \rho\},$ is a given continuous, convex function satisfying a local Lipschitz condition on $B_\rho.$ That is, for $u,v\in B_\rho$  there exists a continuous function $L_f(\cdot)\,:\,\RR^+\to\RR^+$ such that
\bb{lip}
\|f(u)-f(v)\|_{\mathbb{H}^s(\Omega)}\le L_f(c)\|u-v\|_{\mathbb{H}^s(\Omega)}.
\ee
The norm $\|\cdot\|_{\mathbb{H}^s(\Omega)}$ will be defined in the following section. We further assume that $f$ is a monotonically increasing function on $B_\rho$ and 
\bb{lim}
\lim_{u\to c^-}f(u) = +\infty.
\ee

When $\alpha = s = 1,$ \R{b1} reduces to the following local semilinear problem
\bb{bb1}
\l\{\begin{aligned}
\partial_t u & = -(-\Delta)^s u + f(u), & (x,t) &\in Q_t,\\
u & = 0, & (x,t) &\in \Gamma_T,\\
u(x,0) & = u_0(x), & x &\in \Omega,
\end{aligned}\r.
\ee
which was originally studied by Kawarada \cite{Kawa}. This local problem has been well-studied due to the fact that it models several idealized physical phenomena, including solid-fuel combustion and microelectromechanical systems (MEMS) \cite{Josh1,Kavallaris2008,chan1995thermal,Levine,Josh3,Padgett4,kirk2002quenching}.  For \R{bb1}, it is known that under certain conditions monotonically increasing solutions to the problem may only exist locally \cite{Padgett4,Josh3,Levine2,acker1976quenching}. Further, it is known that for a given function $f,$ the existence of global solutions to \R{bb1} depends on the spatial domain {\em size} and {\em shape} \cite{Chan1,Tian,Padgett4}. We say that two $d-$dimensional spatial domains $\Omega_1$ and $\Omega_2$ have the same shape if there exists $y\in\Omega_1\cap\Omega_2$ and a constant $\lambda>0$ such that
\bb{shape}
\Omega_2 = \{z\,:\,z = y+\lambda(x-y),\ \mbox{for}\ x\in\Omega_1\}.
\ee
Thus, for a {\em fixed domain shape}, determining whether global solutions to \R{bb1} exists, reduces to studying the following steady-state problem
\bb{bbb1}
\l\{\begin{aligned}
\Delta u + \lambda^2f(u) & = 0, & x &\in\Omega,\\
u & = 0, & x &\in\partial\Omega.\\
\end{aligned}\r.
\ee
The existence of a unique positive solution to \R{bbb1} depends on the value of $\lambda,$ and thus, there is a {\em critical domain size} that determines whether the classical Kawarada problems emits a global solution \cite{Chan1}. That is, there is a $\lambda_*>0$ associated to $\Omega$ such that if $0<\lambda<\lambda_*,$ then the solution exists globally, and if $\lambda_*<\lambda<\infty,$ then there exists a time $T_*<\infty$ such that the maximal interval of existence for the solution is $[0,T_*).$ In this latter case, the solution is said to {\em quench in finite time}. For the case when $\lambda = \lambda_*,$ we say that the solution {\em quenches in infinite time}, as $T_* = \infty.$

The purpose of this current study is to extend some of the existing results for \R{bb1} to the nonlocal problem \R{b1}. We note that this extension is not simply an interesting mathematical problem, but is motivated by numerous physical applications. That is, there have been numerous recent works which outline the importance of fractional and nonlocal models in the accurate modeling of multiphysics problems exhibiting anomalous diffusion \cite{chen2010anomalous,hilfer2000applications,podlubny1998fractional,sabatier2007advances}. In particular, solid-fuel combustion has been shown to behave in a nonlocal manner, thus necessitating the need for considerations of such mathematical models \cite{Pagnini2011}.

In order to study the problem \R{b1}, we introduce the following definition of quenching in the nonlocal setting.

\begin{define}
A solution $u$ of \R{b1} is said to {\em quench in finite time} if there exists $T_*<\infty$ such that 
\bb{a1}
\max\l\{\|u(x,t)\|_\infty\,:\,x\in \overline{\Omega}\r\}\rightarrow c^{-}~\mbox{as}~t\rightarrow 
T_*^{-}.
\ee
If \R{a1} holds for $T_*=\infty,$ then $u$ is said to {\em quench in infinite time}. $T_*$ is referred to as the {\em quenching time\/}. The set $\Omega_c\subseteq \Omega$ containing all quenching points is called the {\em quenching set}.
\end{define}

The paper is organized as follows. In the following section we introduce some important mathematical preliminaries, which are vital to the current study. In Section 3 we consider properties of the operators which generate the solution to \R{b1}. In Sections 4 and 5 we determine conditions under which there exists unique continuous solutions to \R{b1} that are both positive and monotonically increasing on the domain of existence. Section 6 is concerned with establishing conditions under which quenching occurs. Section 7 provides some computational experiments to validate the results and provide further insight into the quenching phenomenon. Finally, Section 8 provides concluding remarks regarding the current work.

\section{Mathematical Preliminaries} \clearallnum

We now introduce some basic facts and definitions from fractional calculus. In the following, we let $I\mathrel{\mathop:}=(0,T)$ and $\Gamma(\cdot)$ be Euler's gamma function. Further, for $\alpha> 0$ we define the following function
\bb{gk}
g_\alpha(t) = \l\lbrace\begin{array}{ll}
t^{\alpha-1}/\Gamma(\alpha), & t>0,\\
0, & t\le 0,
\end{array}\r.
\ee
with $g_0(t)\equiv 0.$

\begin{define}
Let $v\in L^1(I)$ and $\alpha\ge 0.$ The Riemann-Liouville fractional integral of order $\alpha$ of $v$ is defined as
$$J^\alpha_t v(t) \mathrel{\mathop:}= (g_\alpha * v)(t) = \int_{0}^t g_\alpha(t-s)u(s)\,ds,\quad t>0,$$
where $J^0_tv(t) = v(t).$
\end{define}

\begin{define}
Let $v\in C^{m-1}(I)$ and $g_{m-\alpha}*v\in W^{m,1}(I)$ where $m\in\NN$ and $0\le m-1<\alpha\le m.$ The Riemann-Liouville fractional derivative of order $\alpha$ of $v$ is defined by 
$${_RD_t^\alpha}v(t) \mathrel{\mathop:}= D_t^m(g_{m-\alpha}*v)(t) = D_t^m J_t^{m-\alpha}v(t),\quad t>0,$$ 
where $D_t^m \mathrel{\mathop:}= d^m/dt^m.$
\end{define}

\begin{define}
Let $v\in C^{m-1}(I)$ and $g_{m-\alpha}*v\in W^{m,1}(I)$ where $m\in\NN$ and $0\le m-1<\alpha\le m.$ The regularized Caputo fractional derivative of order $\alpha$ of $v$ is defined as
\bb{caputo}
\partial_t^\alpha v(t) \mathrel{\mathop:}= D_t^m J_t^{m-\alpha} \l(v(t) - \sum_{i=1}^{m-1}f^{(i)}(0)g_{i+1}(t)\r),\quad t>0,
\ee
where $D_t^m \mathrel{\mathop:}= d^m/dt^m.$ If $v$ is continuously differentiable with respect to $t,$ then $\partial_t^\alpha \to d^m/dt^m$ as $\alpha \to m.$
\end{define}

We note that for $\alpha\in (0,1),$ if $v$ is smooth enough, the Caputo fractional derivative can be written as
$$\partial_t^\alpha v(t) = \frac{1}{\Gamma(1-\alpha)}\int_0^t (t-s)^{-\alpha}v'(s)\,ds.$$

We now summarize some useful properties from fractional calculus in the following lemma \cite{podlubny1998fractional,sabatier2007advances,WANG2012202}.

\begin{lemma}
Let $\alpha,\beta>0.$ Then the following properties hold.
\begin{itemize}
\item[i.] $J_t^\alpha J_t^\beta v = J_t^{\alpha+\beta}v$ for all $v\in L^1(I).$ That is to say, the Riemann-Liouville fractional integral has the semigroup property;
\item[ii.] The Caputo fractional derivative is a left inverse of the Riemann-Liouville fractional integral:
$$\partial_t^\alpha J_t^\alpha v = v,\quad\mbox{for all}\ v\in L^1(I),$$
but in general is not a right inverse. In fact, for all $f\in C^{m-1}(I)$ with $g_{m-\alpha}*v\in W^{m,1}(I)$ where $m\in\NN$ and $0\le m-1<\alpha\le m,$ we have
$$J_t^\alpha \partial_t^\alpha v(t) = v(t) - \sum_{i=0}^{m-1}f^{(i)}(0)g_{i+1}(t).$$
\end{itemize}
\end{lemma}

In order to appropriately introduce the fractional Laplacian considered in this paper, we first define some fractional Sobolev spaces of particular importance. For $\Omega\subseteq\RR^d$ with smooth boundary, $\partial\Omega,$ and $s\in(0,1),$ we define the space $H^s(\Omega)$ as
$$H^s(\Omega)\mathrel{\mathop:}= \l\{w\in L^2(\Omega)\,:\,\int_{\Omega}\int_{\Omega}\frac{|w(x)-w(y)|^2}{|x-y|^{d+2s}}\,dy\,dx<\infty\r\},$$
where $x,y\in\Omega,$ and $|x-y|^2\mathrel{\mathop:}=\textstyle\sum_{i=1}^d (x_i-y_i)^2.$
If we define $|\cdot|_{H^s(\Omega)}$ to be the seminorm given by
$$|w|_{H^s(\Omega)}^2 \mathrel{\mathop:}= \int_\Omega\int_\Omega\frac{|w(x)-w(y)|^2}{|x-y|^{d+2s}}\,dy\,dx,$$
then it follows that $H^s(\Omega)$ is a Hilbert space with norm $\|\cdot\|_{H^s(\Omega)}^2 = \|\cdot\|_{L^2(\Omega)}^2 + |\cdot|_{H^s(\Omega)}^2$ \cite{servadei2014spectrum,DINEZZA2012521}. We define the space $H_0^s(\Omega)$ to be the closure of $C_0^\infty(\Omega)$ with respect to the norm $\|\cdot\|_{H^s(\Omega)}.$ Since $\partial\Omega$ is smooth, it follows that we can define the aforementioned fractional Sobolev spaces via interpolation spaces of index $\theta = 1-s$ \cite{DINEZZA2012521}. That is, for $s\in(0,1),$ we have
\bb{interp}
H^s(\Omega) = [H^1(\Omega),L^2(\Omega)]_\theta;
\ee
{\em i.e.}, $H^s(\Omega)$ is the intermediary Banach space between $L^2(\Omega)$ and $H^1(\Omega).$ Further, we can define the spaces $H_0^s(\Omega),\ s\in(0,1),$ to be
\bb{interp1}
H_0^s(\Omega) = [H_0^1(\Omega),L^2(\Omega)]_\theta.
\ee

The definition of the fractional Laplacian on $\RR^d$ is given as follows.

\begin{define}
Let $S(\RR^d)$ be the class of Schwartz functions. Then for any $w\in S(\RR^d),$ we define the fractional Laplacian of order $s\in (0,1)$ of $w$ as
\bb{fraclap}
(-\Delta)^sw(x) = c_{d,s}\lim_{\epsilon\to 0}\int_{|x-y|\ge 0}\frac{w(x)-w(y)}{|x-y|^{d+2s}}\,dy,
\ee
where $c_{d,s}$ is the normalizing constant given by
$$c_{d,s} = \frac{4^s\Gamma(d/2 + s)}{\pi^{d/2}|\Gamma(-s)|}.$$
\end{define}

It is the case that there is not a unique way of extending the definition of the fractional Laplacian to a bounded domain $\Omega.$ An exploration of the properties of different definitions and their relationships can be found in \cite{servadei2014spectrum}. In this study, we adopt a functional calculus definition of the fractional Laplacian via the Dirichlet Laplacian. That is, let $-\Delta\,:\,L^2(\Omega)\to L^2(\Omega)$ be the classical Laplacian with domain $\mbox{Dom}(-\Delta) = \{w\in H^1_0(\Omega)\,:\,\Delta w \in L^2(\Omega)\}.$ It is known that this operator is unbounded, closed, and has a compact inverse. Thus, the spectrum of $-\Delta$ is discrete, positive, and accumulates at infinity. Moreover, there exists $\{\lambda_n,e_n\}_{n\in\NN} \subset \RR^+\times H_0^1(\Omega)$ such that $\{e_n\}_{n\in\NN}$ is an orthonormal basis of $L^2(\Omega)$ and
$$\l\{\begin{aligned}
-\Delta e_n & = \lambda_ne_n, & x &\in\Omega,\\
e_n & = 0, & x &\in\partial\Omega.
\end{aligned}\r.$$

Thus, we can define the fractional Laplacian, for $w\in C_0^\infty(\Omega),$ to be
\bb{specfrac}
(-\Delta)^sw \mathrel{\mathop:}= \sum_{n\in\NN}w_n\lambda_n^se_n,
\ee
where $w_n = \textstyle\int_\Omega we_n\,dx.$ The definition \R{specfrac} can be extended to the Hilbert space
$$\mathbb{H}^s(\Omega)\mathrel{\mathop:}= \l\{w = \sum_{n\in\NN}w_ne_n\in L^2(\Omega)\,:\, \|w\|_{\mathbb{H}^s(\Omega)}^2 = \sum_{n\in\NN}\lambda_n^s|w_n|^2<\infty\r\},$$
via density arguments. Further, we have that
$$\mbox{Dom}((-\Delta)^{s}) = [H_0^1(\Omega),L^2(\Omega)]_\theta,$$
for $\theta = 1-s.$ This gives us the following characterization of the space $\mathbb{H}^s(\Omega),$
$$\mathbb{H}^s(\Omega) = \l\{\begin{aligned}
& H^s(\Omega), & s &\in(0,1/2),\\
& H^{1/2}_{00}(\Omega), & s &= 1/2,\\
& H^s_0(\Omega), & s &\in(1/2,1).
\end{aligned}\r.$$

Finally, we introduce an important class of functions associated with fractional calculus, known as the Mittag-Leffler functions \cite{WANG2012202,podlubny1998fractional}. 

\begin{define}
Let $\alpha,\beta\in\CC$ with $\mbox{Re}(\alpha),\mbox{Re}(\beta)>0.$ Then we may define the generalized Mittag-Leffler function to be
$$E_{\alpha,\beta}(z) \mathrel{\mathop:}= \sum_{n=0}^\infty \frac{z^n}{\Gamma(\beta+\alpha n)} = \frac{1}{2\pi i}\int_{\gamma}\frac{\lambda^{\alpha-\beta}e^\lambda}{\lambda^\alpha - z}d\lambda,$$
where $\gamma$ is a contour which starts at $-\infty$ and encircles the disc $|\lambda|\le |z|^{1/\alpha}$ counterclockwise.
\end{define}

For $0<\alpha<1,$ $\beta>0,$ we have the following asymptotic expansion of $E_{\alpha,\beta}$ as $z\to\infty$
\bb{ml1}
E_{\alpha,\beta}(z) = \l\lbrace\begin{aligned}
& \alpha^{-1}z^{(1-\beta)/\alpha}\mbox{exp}(z^{1/\alpha})+\varepsilon_{\alpha,\beta}(z), & |\mbox{arg}\,z| &\le \alpha\pi/2,\\
&\varepsilon_{\alpha,\beta}(z), & |\mbox{arg}\,z| &<(1-\alpha/2)\pi,
\end{aligned}\r.
\ee
where
$$\varepsilon_{\alpha,\beta}(z) = -\sum_{n=1}^{N-1}\frac{z^{-n}}{\Gamma(\beta-\alpha n)} + \OO\l(|z|^{-N}\r),\quad\mbox{as}\ z\to\infty.$$
From the above, for any $\omega\in\CC$ we have
$$\partial_t^\alpha E_{\alpha,1}(\omega t^\alpha) = \omega E_{\alpha,1}(\omega t^\alpha) \quad\mbox{and}\quad J_t^{1-\alpha}\l(t^{\alpha-1}E_{\alpha,\alpha}(\omega t^\alpha)\r) = E_{\alpha,1}(\omega t^\alpha).$$
That is, the function $E_{\alpha,1}$ is the eigenfunction corresponding to the Caputo fractional derivative. Thus, $E_{\alpha,\beta}$ may be viewed as a generalization of the standard exponential function in the integer derivative case. This is further supported by the fact that $E_{1,1}(z) = e^z.$ Consider also the Wright function given by
$$\Psi_\alpha(z) \mathrel{\mathop:}= \sum_{n=0}^\infty \frac{(-z)^n}{n!\Gamma(1 - (n+1)\alpha)} = \frac{1}{\pi}\sum_{n=1}^\infty \frac{(-z)^n}{(n-1)!}\Gamma(n\alpha)\sin(n\pi\alpha),\quad z\in\CC,$$
with $0<\alpha<1.$ Then we have the following lemma.

\begin{lemma}
Let $0<\alpha<1.$ For $-1<r<\infty,$ $\lambda>0,$ the following results hold.
\begin{itemize}
\item[i.] $\Psi_\alpha(t)\ge 0,$ $t>0;$
\item[ii.] $\textstyle\int_0^\infty \alpha t^{-\alpha - 1}\Psi_\alpha(t^{-\alpha})e^{-\lambda t}\, dt = e^{-\lambda^\alpha};$
\item[iii.] $\textstyle\int_0^\infty \Psi_\alpha(t)t^r\,dt = \Gamma(1+r)/\Gamma(1+\alpha r);$
\item[iv.] $\textstyle\int_0^\infty \Psi_\alpha(t)e^{-zt}\,dt = E_{\alpha,1}(-z),\ z\in\CC;$
\item[v.] $\textstyle\int_0^\infty \alpha t \Psi_\alpha(t)e^{-zt}\,dt = E_{\alpha,\alpha}(-z),\ z\in\CC.$
\end{itemize}
\end{lemma}

The proof of Lemma 2.2 may be found in \cite{WANG2012202}.
The properties from Lemma 2.2 will be useful in deriving bounds for the operators generating the solution to \R{b1}.

\section{Properties of the Solution Operators} \clearallnum

Throughout this section we assume that $s\in (0,1)$ and we define the Banach space $X = C([0,T],\mathbb{H}^s(\Omega))$ with norm $\|u\|\mathrel{\mathop:}=\textstyle\sup_{t\in[0,T]}\|u(t)\|_{\mathbb{H}^s(\Omega)}.$ Let $\sigma(A)$ and $\rho(A)\mathrel{\mathop:}= \CC-\sigma(A)$ be the spectrum and resolvent set of the operator $A \mathrel{\mathop:}= (-\Delta)^s,$ respectively. It follows from \R{specfrac}, that $-A$ generates a Feller semigroup that has the following Dunford-Riesz representation
\bb{semi1}
T(t) = e^{-tz}(A) = \frac{1}{2\pi i}\int_{\Gamma_\theta}e^{-tz}R(z;A)\,dz,\quad t\in[0,T],
\ee
where $\Gamma_\theta$ is any contour containing $\sigma(A)$ and $R(z,A)$ is the resolvent operator defined as $R(z;A) \mathrel{\mathop:}= (zI-A)^{-1},\ z\in\rho(A).$ We now define the family of operators $\{S_\alpha(t)\}_{t\in[0,T]}$ and $\{P_\alpha(t)\}_{t\in[0,T]}$ to be
$$S_\alpha(t)\mathrel{\mathop:}= E_{\alpha,1}(-zt^\alpha)(A) = \frac{1}{2\pi i}\int_{\Gamma_\theta}E_{\alpha,1}(-zt^\alpha)R(z; A)\,dz,$$
$$P_\alpha(t)\mathrel{\mathop:}= E_{\alpha,\alpha}(-zt^\alpha)(A) = \frac{1}{2\pi i}\int_{\Gamma_\theta} E_{\alpha,\alpha}(-zt^\alpha)R(z;A)\,dz,$$
where $\Gamma_\theta$ is any contour containing $\sigma(A).$ 

We have the following useful properties of the operators $S_\alpha(t)$ and $P_\alpha(t).$

\begin{thm}
For each fixed $t\in[0,T],$ $S_\alpha(t)$ and $P_\alpha(t)$ are linear and bounded operators on $X.$ Moreover, we have the following bounds for all $t\in[0,T],$
\bb{mlbound}
\|S_\alpha(t)\|\le 1,\quad\mbox{and}\quad \|P_\alpha(t)\| \le \frac{1}{\Gamma(\alpha)}.
\ee
\end{thm}

\begin{proof}
The proof follows methods similar to those employed in \cite{WANG2012202}, where we introduce sharper bounds based on the given operator $A.$ We note that the operators are well-defined linear bounded operators on $X$ via \R{ml1}. Thus, we simply need to show that \R{mlbound} holds for all $t\in[0,T].$ For $t\in[0,T]$ we have for any $u\in X,$
\bbb
S_\alpha(t)u &=& \frac{1}{2\pi i}\int_{\Gamma_\theta}E_{\alpha,1}(-zt^\alpha)R(z;A)u\,dz\nnn\\
& = & \frac{1}{2\pi i}\int_0^\infty \Psi_\alpha(\lambda)\int_{\Gamma_\theta}e^{-\lambda zt^\alpha}R(z;A)u\,dz\,d\lambda\nnn\\
& = & \int_0^\infty \Psi_\alpha(s) T(st^\alpha)u\, ds,\label{sbound1}
\eee
by {\em iv.} of Lemma 2.2, {\em ii.} of Lemma 3.1, and Fubini's theorem. Thus, by {\em iii.} of Lemma 2.2 and the fact that $T(t)$ is a contraction, we have
\bb{sbound2}
\|S_\alpha(t)u\| \le \int_0^\infty \Psi_\alpha(s)\|u\|\,ds = \|u\|,\quad t\in[0,T],\ u\in X,
\ee
as desired. Now, an argument similar to the above gives 
\bb{pbound1}
P_\alpha(t) = \int_0^\infty \alpha s\Psi_\alpha(s)T(st^\alpha)u\, ds,\quad t\in[0,T],\ u\in X.
\ee
Once again, by {\em iii.} of Lemma 2.2 and the fact that $T(t)$ is a contraction, we have
\bb{pbound2}
\|P_\alpha(t)u\| \le \alpha\int_0^\infty s\Psi_\alpha(s)\|u\|\,ds \le \frac{\alpha\Gamma(2)}{\Gamma(1+\alpha)}\|u\| = \frac{\|u\|}{\Gamma(\alpha)},
\ee
for $t\in[0,T],\ u\in X,$ as desired. Thus, the estimate \R{mlbound} holds.
\end{proof}

The family of operators $\{S_\alpha(t)\}_{t\in[0,T]}$ and $\{P_\alpha(t)\}_{t\in[0,T]}$ are well-studied families of operators with numerous useful properties. Of particular interest are the studies connecting $\alpha-$resolvent and $(\alpha,\beta)-$resolvent operators to the solution operators of fractional Cauchy problems \cite{LIZAMA2011184}. Other useful properties of these families are outlined in the following lemma.

\begin{lemma}
The operators $S_\alpha(t)$ and $P_\alpha(t)$ have the following properties.
\begin{itemize}
\item[i.] $\{S_\alpha(t)\}_{t\in[0,T]}$ and $\{P_\alpha(t)\}_{t\in[0,T]}$ are strongly continuous.
\item[ii.] For every $t\in[0,T],$ $S_\alpha$ and $P_\alpha$ are compact operators.
\item[iii.] $S_\alpha '(t)u = -t^{\alpha-1}AP_\alpha(t)u$ and $S_\alpha '(t)u$ is locally integrable on $(0,\infty)$ for $u\in X.$
\item[iv.] For all $u\in\mbox{Dom}(A),$ $t\in[0,T],$ we have $\partial_t^\alpha S_\alpha(t)u = -AS_\alpha(t)u.$
\item[v.] For all $t\in[0,T],$ we have $S_\alpha(t) = J_t^{1-\alpha}(t^{\alpha-1}P_\alpha(t)).$
\end{itemize}
\end{lemma}

\no The proof of Lemma 3.1 may be found in \cite{WANG2012202}.

\section{Local Existence and Uniqueness} \clearallnum

In order to investigate the existence and uniqueness of \R{b1} we recast the problem into the setting of a Banach space $X = C([0,T],\mathbb{H}^s(\Omega))$ by considering
\bb{c1}
\l\{\begin{aligned}
\partial^\alpha_t u_t &= -Au_t + f(u_t), & 0<t<T,\\
u_{0} &\in\mathbb{H}^s(\Omega),
\end{aligned}\r.
\ee
where $u_t \mathrel{\mathop:}= u(\cdot, t)$ and $A \mathrel{\mathop:}= (-\Delta)^s\,:\mbox{Dom}(A) = \mathbb{H}^s(\Omega)\subset X \to X.$ We further note that $X$ has the norm $\|u_t\|\mathrel{\mathop:}=\textstyle\sup_{t\in [0,T]} \|u_t\|_{\mathbb{H}^s(\Omega)}.$ We now define our notion of solution when considering \R{c1}.

\begin{define}
A function $u_t\in X$ is a mild solution to \R{c1} if $\|u_t\|<c$ and for any $t\in [0,T]$ 
\bb{sol1}
u_t = S_\alpha(t)u_0 + \int_{0}^t (t-s)^{\alpha-1}P_\alpha(t-s)f(u_s)\,ds.
\ee
\end{define}

We now demonstrate the local existence and uniqueness of the solution to \R{c1} via the Banach fixed point theorem.

\begin{thm}
There exists a $T>0$ such that \R{c1} has a unique mild solution $u_t$ on the interval $[0,T].$
\end{thm}

\begin{proof}
Let $0<r<c$ be fixed and consider the set 
$$B_{r,T} \mathrel{\mathop:}= \l\lbrace u_t\,:\, t\in [0,T),\ \|u_t\|\le r\r\rbrace$$
We note that the set $B_{r,T}$ is a nonempty closed subset of $X.$ Hence, $B_{r,T}$ is a Banach space with norm $\|u_t\|_{B_{r,T}} \mathrel{\mathop:}= \|u_t\|.$ For any $u_t\in B_{r,T}$ we define $F\,:\,B_{r,T}\to X$ as
\bb{op}
(Fu)_t \mathrel{\mathop:}= S_\alpha(t)u_0 + \int_{0}^t (t-s)^{\alpha-1}P_\alpha(t-s)f(u_s)\,ds.
\ee
In order to apply the Banach fixed point theorem, we must show that $F$ is actually a contraction mapping into $B_{r,T}.$ Since $B_{r,T}\subset X,$ it suffices to show that $\|(Fu)_t\|_{B_{r,T}}\le r.$ We first note that
\bb{bound1}
\|S_\alpha(t)u_0\|_{B_{r,T}} \le \|S_\alpha(t)u_0\| \le \|u_0\| \equiv r_0 \ll c
\ee
by \R{mlbound}. Since $f$ is Lipschitz continuous, it follows from \R{lip} that
\bb{growth}
\|f(u)\|_{\mathbb{H}^s(\Omega)} \le L_f(c)(1+\|u\|_{\mathbb{H}^s(\Omega)}),
\ee
for $u\in B_c.$ We then have the following bound
\bbb
&& \l\|\int_{0}^t (t-s)^{\alpha-1}P_\alpha(t-s)f(u_s)\,ds\r\|_{B_{r,T}}\nnn\\
&& ~~~~~~~~~~~~~~~~~~~~~~~~~~~~~~ \le \int_{0}^t |(t-s)^{\alpha -1}|\|P_\alpha(t-s)\|\|f(u_s)\|\,ds\nnn\\
&& ~~~~~~~~~~~~~~~~~~~~~~~~~~~~~~ \le \frac{1}{\Gamma(\alpha)} \int_{0}^t (t-s)^{\alpha-1}L_f(c)\l(1+\|u_s\|_{B_{r,T}}\r)\,ds\nnn\\
&& ~~~~~~~~~~~~~~~~~~~~~~~~~~~~~~ \le \frac{1}{\Gamma(\alpha)}\frac{t^\alpha}{\alpha}L_f(c)(1+c)\nnn\\
&& ~~~~~~~~~~~~~~~~~~~~~~~~~~~~~~ \le \frac{1+c}{\Gamma(1+\alpha)}L_f(c)T^\alpha\label{bound2}
\eee
where we have employed the growth condition \R{growth} and the bound \R{mlbound} from Theorem 3.1. Combining \R{bound1} and \R{bound2} gives
$$\|(Fu)_t\|_{B_{r,T}} \le r_0 + \frac{1+c}{\Gamma(1+\alpha)}L_f(c)T^\alpha \le c$$
by choosing choosing $T\in(0,T_1)$ where 
$$T_1 \mathrel{\mathop:}=  \l(\frac{(c-r_0)\Gamma(1+\alpha)}{L_f(c)(1+c)}\r)^{1/\alpha}.$$
Thus, $F\,:\,B_{r,T}\to B_{r,T}$ on $[0,T].$

Next we must show that $F$ is a contraction in $X.$ To that end we have
$$(Fu)_t - (Fv)_t = \int_{0}^t (t-s)^{\alpha-1}P_\alpha(t-s)\l[f(u_s)-f(v_s)\r]\, ds.$$
It then follows that
\bbbb
\|(Fu)_t-(Fv)_t\|_{B_{r,T}} &\le & \int_{0}^t |(t-s)^{\alpha-1}|\|P_\alpha(t-s)\|\|f(u_s)-f(v_s)\|\, ds\\
& \le & \frac{L_f(c)\alpha}{\Gamma(1+\alpha)}\int_{0}^t (t-s)^{\alpha-1}\|u_s - v_s\|\,ds\\
& \le & \frac{L_f(c)}{\Gamma(1+\alpha)}T^\alpha \|u_t-v_t\|_{B_{r,T}}
\eeee
Thus, $F$ is a contraction on $B_{r,T}$ for $t\in[0,T]$ if 
$$0<T<T_2 \mathrel{\mathop:}= \l(\frac{\Gamma(1+\alpha)}{L_f(c)}\r)^{1/\alpha}.$$

Let $T = \min\{T_1,T_2\}.$ Then we have that $F$ has a unique fixed point in $B_{r,T}$ for $t\in[0,T]$ via the Banach fixed point theorem. Thus, \R{c1} has a unique mild solution on $[0,T].$
\end{proof}

\begin{rem}
We note that the solution can be extended in time as long as $\|u\|<c$ and the solution only ceases to exist once $\|u\|=c.$
\end{rem}

\section{Solution Positivity and Monotonicity} \clearallnum

Classical studies regarding the integer order Kawarada equations have often considered the solution positivity and monotonicy \cite{Josh1,Josh3,Padgett4,Kavallaris2008}. These properties are critical in many situations modeled by these equations, such as solid-fuel combustion \cite{chan1995thermal}. To that end, we determine conditions under which the continuous solution to \R{b1} is positive and monotone on its interval of existence.

Positivity of the solution to \R{b1} is relatively straightforward to verify. We summarize the result with the following lemma.

\begin{lemma}
The solution to \R{b1} given by \R{sol1} is positive on its interval of existence. 
\end{lemma}

\begin{proof}
Recall \R{sbound1}. Then we have the following representation of $S_\alpha u_0:$
$$S_\alpha(t)u_0 = \int_0^\infty \Psi_\alpha(s) T(st^\alpha)u_0\, ds.$$
By {\em i.} of Lemma 2.3, we have that $\Psi_\alpha(s)\ge 0$ and we have that $T(st^\alpha)$ is positive for $t\in[0,T]$ by definition. 
By the assumption that $0\le u_0<c,$ it follows that $S_\alpha(t)u_0$ is nonnegative.

Similarly, by recalling \R{pbound1}, we have the following representation of $P_\alpha:$
$$P_\alpha(t) = \int_0^\infty \alpha s\Psi_\alpha(s)T(st^\alpha)\,ds.$$
A similar argument gives that $P_\alpha(t)$ is a nonnegative operator for $t\in[0,T]$ since $\alpha,s>0.$ Thus, it follows that
$$\int_0^t (t-s)^{\alpha-1}P_\alpha(t-s)f(u_s)\,ds$$
is positive for $t\in[0,T].$ The result follows by the fact that the sum of these operators will preserve positivity.
\end{proof}

\begin{rem}
From Lemma 5.1 we are also able to conclude that $u_0$ is the minimum value that the continuous solution to \R{b1} can attain. When $u_0 \equiv 0,$ the conclusion is clear. If $u_0\ge 0,$ then we may define $w = u-u_0,$ where $u$ solves \R{b1}. Then $w$ satisfies the following problem
\bb{w1}
\l\{\begin{array}{ll}
\partial_t^\alpha w = -(-\Delta)^sw+f(w+u_0), & (x,t)\in Q_T,\\
w = 0, & (x,t)\in\Gamma_T,\\
w(x,0) = 0, & x\in\Omega.
\end{array}\r.
\ee
The continuous solution to \R{w1} is given by
$$w_t = \int_0^t(t-s)^{\alpha-1}P_\alpha(t-s)f(w_s+u_0)\,ds,$$
which gives $w_t \ge 0$ for all $t\in[0,T].$ This means that $u_t\ge u_0$ for all $t\in[0,T],$ which gives that $u_0$ is the minimum of the solution to \R{b1}. In particular, this means that $f(u_0)\le f(u_t)$ for $t\in[0,T],$ since $f$ is monotonically increasing.
\end{rem}

We now derive conditions under which the continuous solution to \R{b1} is monotonically increasing with respect to time on its interval of existence. It is worth noting that in the case when $\alpha = s = 1,$ it has been shown that a sufficient condition to guarantee monotonically increasing solutions is
\bb{inc1}
\Delta u_0 + f(u_0) > 0.
\ee
Classical proofs of this result, however, have required the differentiability of the reaction function $f(u)$ \cite{Levine,Levine2}. The following theorem develops a monotonicity result that does not require the differentiability of the reaction function $f.$

\begin{thm}
Assume that $-(-\Delta)^su_0+f(u_0)>0.$ Then the solution $u$ to \R{b1} is monotonically increasing with respect to time on its interval of existence. That is, the sequence $\{u_t\}_{t\in[0,T)}$ is a strictly increasing sequence of functions.
\end{thm}

\begin{proof}
Let $u_t$ be a continuous solution to \R{c1}, or similarly \R{b1}, given by \R{sol1}. We proceed by considering the difference between $u_t$ and $u_s$ for $t>s\ge 0.$
First, note that 
\bb{mon1}
S_\alpha(t)u_0 - S_\alpha(s)u_0  =  \l(S_\alpha(s) + S_\alpha '(\xi)(t-s)\r)u_0 - S_\alpha(s)u_0.
\ee
for some $\xi\in[s,t].$ By {\em iii.} of Lemma 3.1, we have $S'_\alpha(t)u_0 = -t^{\alpha-1}AP_\alpha(t)u_0.$ Hence, by \R{mon1} we have
\bb{mon2}
S_\alpha(t)u_0 - S_\alpha(s)u_0 = -\xi^{\alpha-1}AP_\alpha(\xi)(t-s)^{\alpha-1}u_0 = -\xi^{\alpha-1}(t-s)^{\alpha-1}P_\alpha(\xi)Au_0,
\ee
where the fact that $A$ and $P_\alpha(t)$ commute for all $t\in[0,T]$ follows by the definition of $P_\alpha.$

We now consider the difference of the integrals in the continuous solutions. By noting that $f(u_t) \ge f(u_0)$ for all $t\ge 0$ (see Remark 5.1), we have
\bbb
&&\int_0^t(t-w)^{\alpha-1}P_\alpha(t-w)f(u_w)\,dw - \int_0^s(s-w)^{\alpha-1}P_\alpha(s-w)f(u_w)\,dw\nnn\\
&&~~~~~~~~~~~~~~~~~~~~ \ge \l(\int_0^t A^{-1}S_\alpha '(t-w)\,dw - \int_0^s A^{-1}S_\alpha(s-w)\,dw\r)f(u_0)\nnn\\
&&~~~~~~~~~~~~~~~~~~~~ = \xi^{\alpha-1}(t-s)^{\alpha-1}P_\alpha(\xi)f(u_0),\label{mon3}
\eee
where the above integration is well defined since $A^{-1}$ is compact.

Thus, combining \R{mon1} and \R{mon3} we have
\bb{mon4}
u_t - u_s \ge \xi^{\alpha-1}(t-s)^{\alpha-1}P_\alpha(\xi)\l(-Au_0 + f(u_0)\r).
\ee
Monotonicity follows from the assumption that $-Au_0 + f(u_0) = -(-\Delta)^su_0+f(u_0)>0.$
\end{proof}

\begin{rem}
We note that the restriction $-(-\Delta)^su_0+f(u_0)$ is not unreasonable. This is clear from the fact that letting $u_0\equiv 0$ results in $-(-\Delta)^su_0+f(u_0) = f(0)>0.$
\end{rem}

\section{Finite Time Quenching}

In this section we show that the local in time solution to \R{b1} cannot always be extended to a global in time solution. This fact is well established in the cases for $\alpha = s = 1$ and also recently for $s = 1$ and $\Omega = [0,1]\subset\RR$ \cite{Kawa,Kavallaris2008,Levine,Levine2,Sawangtong2017}. We now show that this result holds in a similar manner for any $\alpha,s\in(0,1).$ 
The key is demonstrating a relationship between the existence of solutions to \R{b1} and weak solutions to the following stationary problem
\bb{ss1}
\l\{\begin{array}{ll}
(-\Delta)^sv = f(v), & x\in\Omega,\\
v = 0, & x\in\partial\Omega,
\end{array}\r.
\ee
such that $v\ge u_0.$
A function $v$ is called a {\em weak solution} to \R{ss1} if
\bb{weak}
v(x) = \int_\Omega G_{\Omega,s}(x,y)f(v(y))\,dy\quad\mbox{and}\quad v(x)\ge u_0(x),
\ee
for $x-$a.e. in $\Omega,$
where $G_{\Omega,s}$ is the Green's function associated to the spectral fractional Laplacian with zero Dirichlet conditions on $\Omega$ \cite{chen2010heat,felmer2014radial}.


Let $B\subset \Omega$ be the largest open ball contained in $\Omega.$ Without loss of generality, we further assume that $B$ is centered at the origin. We begin by establishing results for \R{b1} and \R{ss1} on $B.$ We will proceed in a manner similar to \cite{Levine2}.

\begin{lemma}
Let $B$ be as above and consider \R{b1} and \R{ss1} on $B.$
A solution to \R{b1} exists globally if and only if there exists a weak stationary solution to \R{ss1}. Moreover, in this case the solution to \R{b1} approaches the minimum solution to \R{ss1} monotonically from below as $t\to \infty.$
\end{lemma}

\begin{proof}
Suppose that there is a solution $v$ of \R{ss1}. 
We begin by assuming that $u_0(x)$ is symmetric and monotonically decreasing about the origin in $B.$
Thus, it follows that the solution $u$ of \R{b1} and $v$ of \R{ss1} are both radially symmetric and monotone decreasing about the origin \cite{felmer2014radial}. Then there is a maximum of both $u$ and $v$ occurring at $x = 0.$ Let $w\mathrel{\mathop:}= v-u$ and consider
\bb{ww1}
\l\{\begin{array}{ll}
\partial_t^\alpha w \ge -(-\Delta)^s w + d(x,t)w, & (x,t)\in B\times (0,T),\\
w = 0, & (x,t)\in\partial B\times (0,T),\\
w(x,0) = v(x) - u_0(x), & x\in B,\\
w(0,t) \ge 0, & t\in [0,T],
\end{array}\r.
\ee
where $d$ is a bounded function resulting from the convexity of the function $f.$
By \R{sol1} and Lemma 5.1 we can conclude that $w\ge 0$ for $x\in\Omega\backslash\{0\}$ and $t\in [0,T).$

We will show that $u$ must exist globally if $v$ exists. Note that $\partial_t^\alpha w = -\partial_t^\alpha u < 0,$ which implies that $\partial_t w < 0.$ It can be shown that since $u$ and $v$ both have a maximum at $x=0,$ then $w$ also is maximized at $x=0.$ For the sake of a contradiction, assume that there is a $T<\infty$ such that the maximal interval of existence of $u$ is $0\le t<T.$ By Theorem 4.1, this implies that $\textstyle\lim_{t\to T^-}u(x,t) = c.$ Moreover, since $\partial_tw<0$ and $w\ge 0,$ the fact that $w$ is maximized at $x=0$ gives $\textstyle\lim_{t\to T^-}u(x,t) = v(x)$ uniformly on $\Omega.$ Consider $\tilde{B}\mathrel{\mathop:}= B\backslash\{0\}\times [0,T].$ It follows that $w\ge 0$ on $\tilde{B}$ and satisfies \R{ww1} with $w(x,T) = 0.$ By considering \R{sol1} and Theorem 4.1 again, we conclude that $w(x,t)>0$ for $(x,t)\in\tilde{B}.$ This is a contradiction, and thus, $T=\infty$ and $u$ is a global solution to \R{b1} on $B.$

Now let $u_0$ be arbitrary and nonnegative. We can then choose a function $\tilde{u}_0\in\mathbb{H}^s(\Omega)$ such that $\tilde{u}_0$ is symmetric and radially decreasing in $B,$  $0\le\tilde{u}_0\le u_0$ in $B,$ and $\tilde{u}_0\equiv 0$ on $\partial B.$ A modification of the above arguments by considering \R{b1} and \R{ss1} with $\tilde{u}_0$ gives a similar conclusion regarding the existence of a global solution to \R{b1} if a solution to \R{ss1} exists. 

Now assume that $u$ is a global solution to \R{b1}. Let 
$$Z(x,t) \mathrel{\mathop:}=\int_{\Omega} G_{B,s}(x,y)u(y,t)\,dy,$$ 
where $G_{B,s}$ is the Green's function associated to \R{ss1}. 
Then we have
\bbb
\partial_t^\alpha Z(x,t) & = & \int_{B}G_{B,s}(x,y)\partial_t^\alpha u(y,t)\,dy\label{green2}\\
& = & \int_{B}G_{B,s}\l[-(-\Delta)^su(y,t) + f(u(y,t))\r]\,dy\nnn\\
& = & -u(x,t) + \int_{B}G_{B,s}(x,y)f(u(y,t))\,dy,\label{green1}
\eee
where \R{green1} is valid for $(x,t)\in B\times (0,T)$ as long as $u(x,t)<c,$ and the integration by parts formula for the fractional Laplacian has been applied to obtain \R{green1}. Since $f$ and $\partial_t^\alpha u$ are both monotonically increasing, it follows that the expression in \R{green1} converges and can be expressed as
$$Y(x) \mathrel{\mathop:}= -v(x) + \int_\Omega G_{\Omega,s}(x,y)f(v(y))\,dy,$$
where $v(x) = \textstyle\lim_{t\to\infty}u(x,t)\le c.$ From \R{green2} we have that $Y\ge 0.$ Note that if $Y>0,$ it would follow that $Z\to \infty$ as $t\to\infty,$ and hence, $u$ would reach $c$ in finite time. Since this would be a contradiction to our assumption that $u$ exists globally, we have that $Y(x)\equiv 0.$ Thus, after rearranging the expression for $Y$ we have
$$v(x) = \int_\Omega G_{\Omega,s}(x,y)f(v(y))\,dy,$$
and $v$ is a weak solution to \R{ss1}. This conclusion gives the desired result.
\end{proof}

Lemma 6.1 only gives an equivalence between global solutions and stationary solutions on open balls. The following result will aid in extending this result to arbitrary domains.

\begin{lemma}
Let $\Omega\subset\RR^d$ be a bounded convex domain with smooth boundary $\partial\Omega$ and let $B$ be an open ball such that $B\subset \Omega.$ Let $u_1$ and $u_2$ be solutions of \R{b1} on $B\times(0,T_1)$ and $\Omega\times(0,T_2),$ respectively. Further, let $T_0\mathrel{\mathop:}=\min\{T_1,T_2\}$ and $u_0\equiv 0.$ Then $u_1<u_2$ for $(x,t)\in B\times (0,T_0).$
\end{lemma}

\begin{proof}
Without loss of generality, assume $B\subset\Omega$ is an open ball centered at the origin. Define $w\mathrel{\mathop:}= u_2 - u_1.$ Then we have
\bb{ball1}
\l\{\begin{array}{ll}
\partial_t^\alpha w \ge -(-\Delta)^s w + d(x,t)w, & (x,t)\in B\times(0,T_0),\\
w>0, & (x,t)\in\partial B\times (0,T_0),\\
w(x,0) = 0, & x\in B,
\end{array}\r.
\ee
where $d$ is bounded and positive in $B.$ By \R{sol1} it follows that $w\ge 0.$ By \R{ball1} it follows that $w>0,$ and thus, $u_2>u_1$ for $(x,t)\in B\times(0,T_0).$ 
\end{proof}

It follows that a solution to \R{b1} can only {\em quench} if the domain is ``large" enough. This result is summarized by the following theorem.

\begin{thm}
If there exists an open ball such that $B\subset \Omega$ and \R{ss1} does not exist on $B,$ then the solution to \R{b1} quenches.
\end{thm}

\begin{proof}
Without loss of generality, assume that $B\subset\Omega$ is an open ball containing the origin. The result follows immediately from Lemmas 6.1 and 6.2.
\end{proof}

\section{Conclusions and Future Work}

This article has studied a time-space fractional semilinear equation which is the generalization of the standard Kawarada equation. Herein, the fractional Caputo derivative is used in time, while the fractional Laplacian is considered in space. It is shown that the standard properties exhibited by the local Kawarada equation may be extended to the nonlocal equation of interest. In particular, under appropriate restrictions, the solution to the fractional Kawarada equation is positive and monotonically increasing on its domain of existence. Moreover, it is shown that there are conditions under which the solution will {\em quench} in finite time. These conditions depend on the domain size and shape, as well as the nonlinear reaction term.

The current study has only considered the theoretical aspects of this problem, but it is well known that the numerical and computational aspects of fractional problems yield even more difficult hurdles. While these issues are not considered herein, they are of interest and will be the topic of forthcoming works. It is known that the nonlocal nature of the problem and quenching phenomenon must be treated carefully while also avoiding unnecessary memory issues during computations. Moreover, it will be of interest whether splitting methods, such as the Alternating Direction Implicit Method, can be employed to improve efficiency and accuracy in such nonlocal problems. This direction of research is still in its infancy and will likely provide interesting mathematical problems for years to come.

\section*{References}

\bibliographystyle{plain}
\bibliography{Frac_Bib1}






\end{document}